\documentclass{article}

\usepackage{microtype}
\usepackage{graphicx}
\usepackage{subfigure}
\usepackage{booktabs} 
\usepackage{amsmath,amscd,amsfonts,accents}
\usepackage{amssymb}
\usepackage{float}
\usepackage{enumerate,tabularx}
\usepackage{amsthm}
\newtheorem{theorem}{Theorem}
\newtheorem{corollary}{Corollary}[theorem]
\newtheorem{proposition}{Proposition}

\newtheorem{example}{Example}
\newtheorem{remark}{Remark}[theorem]
\usepackage{lipsum}
\usepackage{mathtools}
\usepackage{cuted}
\newcommand{\beginsupplement}{%
        \setcounter{table}{0}
        \renewcommand{\thetable}{S\arabic{table}}%
        \setcounter{figure}{0}
        \renewcommand{\thefigure}{S\arabic{figure}}%
     }
\usepackage{amsmath,amscd,amsfonts,accents}
\newcommand{\interior}[1]{\accentset{\smash{\raisebox{-0.12ex}{$\scriptstyle\circ$}}}{#1}\rule{0pt}{2.3ex}}
\usepackage{hyperref}

\usepackage[accepted]{icml2020}

\icmltitlerunning{Transformation of ReLU-based recurrent neural networks from discrete-time to continuous-time}

\begin{document}

\twocolumn[
\icmltitle{Transformation of ReLU-based recurrent neural networks from discrete-time to continuous-time}


\icmlsetsymbol{equal}{*}

\begin{icmlauthorlist}
\icmlauthor{Zahra Monfared}{ger}{*}
\icmlauthor{Daniel Durstewitz}{ger,phys}{*}
\end{icmlauthorlist}

\icmlaffiliation{ger}{Department of Theoretical Neuroscience, Central Institute of Mental Health, Medical Faculty Mannheim, Heidelberg University, Mannheim, Germany}
\icmlaffiliation{phys}{Faculty of Physics and Astronomy, Heidelberg University, Heidelberg, Germany}

\icmlcorrespondingauthor{Zahra Monfared}{zahra.monfared@zi-mannheim.de}
\icmlcorrespondingauthor{Daniel Durstewitz}{daniel.durstewitz@zi-mannheim.de}


\icmlkeywords{dynamical systems, state space analysis, piecewise linear recurrent neural networks, ordinary differential equations, chaos, time series}

\vskip 0.3in
]



\printAffiliationsAndNotice{\icmlEqualContribution} 
\begin{abstract}
Recurrent neural networks (RNN) as used in machine learning are commonly formulated in discrete time, i.e. as recursive maps. This brings a lot of advantages for training models on data, e.g. for the purpose of time series prediction or dynamical systems identification, as powerful and efficient inference algorithms exist for discrete time systems and numerical integration of differential equations is not necessary. On the other hand, mathematical analysis of dynamical systems inferred from data is often more convenient and enables additional insights if these are formulated in continuous time, i.e. as systems of ordinary (or partial) differential equations (ODE). Here we show how to perform such a translation from discrete to continuous time for a particular class of ReLU-based RNN. We prove three theorems on the mathematical equivalence between the discrete and continuous time formulations under a variety of conditions, and illustrate how to use our mathematical results on different machine learning and nonlinear dynamical systems examples.
\end{abstract}
\section{Introduction}\label{sc-intro}
Recurrent neural networks (RNN) are popular devices in machine learning and AI for tasks that require processing and prediction of temporal sequences, like machine translation \cite{sus114}, natural language processing \cite{kum,zah}, or tracking of moving objects in videos \cite{mil}. More recently, in the natural sciences, biology and physics in particular, RNNs were also introduced as powerful tools for approximating the unknown nonlinear dynamical system (DS) that produced a set of empirically observed time series, i.e. for identifying the data-generating nonlinear DS in a completely data-driven, bottom-up way \cite{dur1,kop219,raz,vla,zha}. Theoretically, it has been proven that (continuous) RNN can approximate the flow field of any other nonlinear DS to arbitrary precision on compact sets of the real space under some mild conditions \cite{fun,han,kim,tri}.\\ \\
RNNs, in the form most widely used in machine learning, constitute discrete-time DS defined by a recursive transition rule (difference equation) which maps the network's activation states among consecutive time steps, $z_t=F_{\theta}(z_{t-1}, s_t)$, where $\theta$ are parameters of the system and $\{s_t\}$ is a sequence of external inputs. This formulation is highly advantageous for training RNN on observed data sequences since efficient variational inference and Expectation-Maximization algorithms, which do not require numerical integration of nonlinear ODE, exist for discrete-time systems \cite{dur1,kop219,raz,zha}. In many scientific contexts, like neuroscience \cite{koc}, ODE systems are often highly nonlinear and stiff and thus would require more involved implicit numerical integration schemes to achieve accurate solutions \cite{koc,oza,pre}. Furthermore, empirical data are always sampled at \textit{discrete} time points, such that in most cases this assumption may not be too limiting for the purpose of model inference.\\ \\
On the other hand, natural systems evolve in continuous time, and hence most mathematical theories in physics and biology are formulated in \textit{continuous} time. Thus, for the purpose of theoretical analysis of models inferred from experimental data a continuous-time formulation of the system with biologically or physically meaningful time constants would be preferable. Moreover, a continuous time ODE system enables to analyze properties of the DS under study that are much more difficult or impossible to assess in discrete time. For instance, a continuous time DS comes with a flow field that enables to visualize the system’s dynamics more easily (e.g. Fig.\ref{figure3}), and it enables to smoothly interpolate between observed data points and thus to assess the solution at arbitrary time points. This is of advantage in particular when observations come at irregular event times \cite{che}. Also, separating DS into slow and fast subsystems (separation of time scales) is a powerful analysis tool \cite{dur4,ros} that is not readily available for discrete-time systems.
More generally, the fact that an ODE system is usually smooth almost everywhere eases the mathematical study of many phenomena, like those of periodic and non-periodic solutions, stable and unstable manifolds of fixed points, bifurcations, or stability of solutions more generally \cite{abs}. In fact, finding explicit solutions is often impossible even for 1-dimensional recursive maps, such that one often has to revert to graphical methods like cobwebs \cite{abs,has}.\\ \\
It would therefore be highly desirable to find ways in which discrete time RNN models could be embedded into continuous time systems without changing their phase space. Such an embedding requires that for a given discrete-time system $\{\phi_{t} \}_{t \in \mathbb{T}_d}$ there exists a continuous-time system $\{\psi_{t}\}_{t \in \mathbb{T}_c}$  such that $\phi_{t} = \psi_{t}$ for $t \in \mathbb{T}_{d}$ \cite{has}. In general, finding such an embedding is not possible for nonlinear discrete-time systems, while the reverse problem, although not trivial \cite{oza}, is much easier and there are different ways for obtaining a discrete time system from a continuous one\footnote{The most important of such methods is the Poincar\'e map, where a continuous time will be reduced to a discrete time system by successive intersections of the flow of the continuous system with the Poincar\'e section, such that most dynamical properties of the original system will be preserved.}. The present paper will address this issue using a specific formulation of RNNs, namely piecewise-linear RNNs (PLRNN) that employ rectified-linear units (ReLU) as their activation function. PLRNN are universal in terms of their dynamical repertoire \cite{koi,sie,zhou}, can extract long-term dependencies in sequential data just like LSTMs can \cite{sch}, and – in particular – have been used previously to infer nonlinear DS from time series data \cite{dur1,kop219}. We show that for this specific class of RNN models mathematically equivalent ODE systems, in the sense defined above, can be derived under almost all conditions. 
We will exemplify these results on a couple of machine learning and DS model systems, including 
an ODE solution to the well-known ‘addition problem’ \cite{hoch}, limit cycle and chaotic dynamics, and on a PLRNN inferred from empirical time series (human functional magnetic resonance imaging [fMRI] data; \cite{kop219}). 
\section{Related work}\label{sec-1}
In some recent papers \cite{bro,jor} continuous-time ODE ‘approximations’ of discrete RNN were sought based on ‘inverting’ the forward Euler rule for numerically solving continuous ODE systems. A related idea is that of ‘Neural ODE’ \cite{che}, where the flow is given by a (deep) neural network 
 (cf. \cite{pear}) to yield a system continuous across ‘space’ (layers) or time (see also \cite{aba} for closely related ideas). In \cite{chan} the reverse approach is taken for obtaining a discrete time formulation that preserves certain properties of the ODE system. None of this work, however, explicitly considered the problem of finding an equivalent continuous-time description of a discrete-time RNN. In fact, apart from the fact that a simple forward Euler rule is known to be rather inaccurate and unstable for integrating stiff ODE systems \cite{pre}, a na\"ive 'inversion' will generally \textit{not} result in a mathematically equivalent ODE system in the sense defined further above, i.e. the resulting system usually will not have the same phase space and temporal behavior. Consequently, much of the previous work was less aimed at finding a mapping between discrete and continuous time networks, but rather to formulate the problem of neural network training in continuous time and/or space to begin with to exploit advantages of an ODE formulation in one way or the other.\\ \\
Rather, a ‘true’ translation of a discrete into a continuous RNN may be achieved by taking the continuous time limit of $x_t=F(x_{t-1})$, $\lim_{dt\to\ 0} \big[\frac{x_{t}-x_{t-1}}{dt} = \frac{F(x_{t-1})-x_{t-1}}{dt} \big]$, which, however, can be highly nontrivial or impossible for nonlinear systems \cite{oza}. \cite{oza}, while focusing on the continuous-to-discrete case, also briefly discusses some ideas on this reverse direction of a discrete-to-continuous mapping for nonlinear DS. Only approximations are considered, however, that may work well only for certain cases (RNN, in particular, were not discussed), while here we seek exact equivalence according to the definition further above. 
In \cite{nuk} a discrete logistic equation is transformed to a continuous-time system by taking the temporal limit. It is, however, not possible to apply such a method for all discrete DS, and even if a transformation is possible, the discrete DS may have different dimensions than the continuous time equivalent; for instance, chaotic behavior is possible in a 1-dimensional recursive map but requires 3 dimensions in an ODE system \cite{stro}. Specifically for \textit{linear} systems, the necessary and sufficient conditions for embedding a discrete-time homogeneous linear system in a continuous-time system have been studied in \cite{rei}. Here, we are interested more generally in discrete-time \textit{non-homogeneous} systems which are \textit{piecewise linear}, i.e. piecewise linear (ReLU-based) recurrent neural networks (PLRNN). We will show how to embed a PLRNN into an equivalent continuous-time ODE system, and how the dynamics of the PLRNN is directly connected to the dynamics of the corresponding ODE system. To the best of our knowledge this is the first time a method is introduced for converting discrete into continuous time RNNs in a mathematically exact way, i.e., such that the systems are mathematically and dynamically equivalent without using any approximations or numerical techniques.
\section{Preliminaries}\label{sec-2}
In the following we will collect some results which we will need for our derivations.

\begin{theorem}\label{pre-thm-1}
Consider the non-homogeneous system 
\begin{align}\label{pre-1}
\Dot{x} \, = \, A x +b.   
\end{align}
Then $\phi(t) = e^{At}$ (with $\phi(0) = I$) is the fundamental matrix solution for linear system $\Dot{x} \, = \, A x$, and the solution of system \eqref{pre-1} has the form
\begin{align}\label{pre-2}
x(t) \, = \, e^{At} \, x_0 + e^{At} \int_{0}^{t}  e^{-A \tau}\, b\, d \tau , \hspace{1cm} x(0)=x_0.
\end{align}
\end{theorem}
\begin{proof}
See \cite{per}.
\end{proof}
\begin{proposition}\label{pro-1}
The matrix $B= \int_{0}^{T} e^{At} \, dt$ is invertible iff for every eigenvalue $\lambda$ of matrix $A$, we have: $\lambda T \notin 2i \pi \mathbb{Z}^*$ ($\mathbb{Z}^* = \mathbb{Z} \setminus \{0\}$).
\end{proposition}
\begin{proof}
The Taylor expansion of matrix $B$ has the form:
\begin{align}
B \, = \, IT + A \frac{T^2}{2!}  +  A^2 \frac{T^3}{3!} +   A^3 \frac{T^4}{4!} + \cdots,
\end{align}
and 
\begin{align}\label{4}
\text{Spectrum}(B) \, = \, \big\{ s(\lambda) \, | \, \lambda \in \text{Spectrum}(A) \big\},
\end{align}
such that 
\begin{align}\label{5}
s(\lambda) \, = \, 
\begin{cases} 
\frac{e^{\lambda T}-1}{\lambda}; \hspace{.5cm} \lambda \neq 0 \\
T; \hspace{1.2cm} \lambda =0
\end{cases}.
\end{align}
$B$ is invertible iff it does not have any zero eigenvalue. So, by \eqref{4}-\eqref{5}, $B$ is invertible iff $\lambda T \notin 2i \pi \mathbb{Z}^*$.
\end{proof}
\textit{\textbf{Logarithm of real matrices.}}
For a complex matrix a logarithm (not necessarily unique) will exist iff it is invertible \cite{hig}. Real matrices do not always have a real logarithm. However, the following theorem guarantees the existence of a real logarithm for a real matrix.
\begin{theorem}\label{pre-thm-4}
A real matrix $A \in \mathbb{R} ^{n \times n}$ has a real logarithm if and only if
\begin{itemize}
    \item[(I)] A is invertible, and
    \item[(II)] every $k \times k$ Jordan block associated with a negative eigenvalue occurs an even number of times in the Jordan form of $A$.
\end{itemize}
\end{theorem}
\begin{proof}
See \cite{nun,she}.
\end{proof}
\begin{corollary}\label{pre-cor-1}
Due to theorem \ref{pre-thm-4} a real nonsingular $2 \times 2$ matrix $A$ with two eigenvalues $\lambda_1$ and  $\lambda_2$ will have a real logarithm in the following cases:
\begin{itemize}
    \item[(1)] $\lambda_1$ and  $\lambda_2$ are complex conjugate, i.e. $\lambda_{1,2}= a \pm ib, b\neq 0$. In this case $A$ and $log(A)$ have the Jordan forms
    $ \begin{pmatrix}
     a & & b \\
     -b & & a
     \end{pmatrix}$ and
     $\begin{pmatrix}
    \alpha & & \beta \\
     -\beta & & \alpha
     \end{pmatrix}$, respectively, where $e^{\alpha \pm i\beta}= a \pm ib$.
    \item[(2)] $\lambda_1$ and  $\lambda_2$ are both real and positive.
    \item[(3)] $\lambda_1 = \lambda_2 = -\lambda$ with $\lambda>0$. In this case $A$ and $log(A)$ have the Jordan forms
     $\begin{pmatrix}
     -\lambda & & 0 \\
     0 & & -\lambda
     \end{pmatrix}$ and 
     $\begin{pmatrix}
    log(\lambda) & & \pi \\
     -\pi & & log(\lambda)
     \end{pmatrix}$, respectively.
\end{itemize}
For more details see \cite{nun,she}. 
\end{corollary}
\begin{remark}
Suppose that $A$ is a real nonsingular $2 \times 2$ matrix which has two equal and negative eigenvalues $\lambda_1 = \lambda_2<0$. If $A$ has the Jordan form 
$\begin{pmatrix}
     \lambda_1 & & 1 \\
     0 & & \lambda_2
     \end{pmatrix}$, then it will not have a real logarithm.
\end{remark}\label{pre-rem-4}
\section{Conversion of discrete- into continuous-time PLRNN}\label{sec-3}
\subsection{Discrete-time RNN model}\label{subsec-1}
Consider a piecewise-linear RNN (PLRNN) of the generic form
\begin{align}\label{eq-1}
 Z_{t+1} \, = \,  A \, Z_{t}+W \phi(Z_{t})+ \, h  
\end{align}
where $\phi(Z_{t})=\max(Z_{t}, 0)$ is the element-wise rectified linear unit (ReLU) transfer function, $Z_ t=(z_{1t}, \cdots, z_{Mt} ) ^T \in \mathbb{R}^{M}$ denotes the neural state vector at time $t =1 \cdots T $, the diagonal entries of $A =diag(a_{11}, \cdots, a_{MM} )\in \mathbb{R} ^{M\times M}$ represent (linear) auto-regression weights, $W  \in \mathbb{R} ^{M\times M}$ is a matrix of connection weights (sometimes assumed to be off-diagonal, i.e. with diagonal elements equal to zero, e.g. \cite{kop219}), and $h$ is a bias term \cite{dur1,kop219}.
\\

The discrete-time PLRNN \eqref{eq-1} can be represented in the form 
\begin{align}\label{eq-1-1}
 Z_{t+1} \, = \, (A +W D_{\Omega(t)}) Z_{t} + \, h, 
\end{align}
where 
$$D_{\Omega(t)}:= \text{diag}(d_{\Omega(t)}),$$ with 
$$d_{\Omega(t)} \, := \,\big(d_{1}(t), d_{2}(t), \cdots, d_{M}(t) \big),$$ 
such that $d_{i}(t)=0$ if $z_{it} \leq 0$ and $d_{i}(t)=1$ if $z_{it} > 0$, for $i=1,2, \cdots, M$. There are $2^{M}$ different configurations for matrix $D_{\Omega(t)}$, depending on the sign of the components of $Z_ t$. That is, the phase space of system \eqref{eq-1-1} is separated into $2^M$ sub-regions by $M2^{M-1}$ hyper-surfaces which form discontinuity boundaries. Now, indexing the $2^M$ different configurations of $D_{\Omega(t)}$ as $D_{\Omega^k}$ for $k \in \lbrace 1,2, \cdots, 2^{M} \rbrace$, we define $2^{M}$ matrices
\begin{align}\label{eq-8}
W_{\Omega^k}:= A +W D_{\Omega^k},    
\end{align}
such that in each sub-region the dynamics are governed by a different linear map (cf. Fig. \ref{figure7}), i.e.
\begin{align}\label{eq-9}
 Z_{t+1} \, = \, W_{\Omega^k} \, Z_{t} + \, h, \hspace{1cm} k \in\lbrace 1, 2, \cdots, 2^{M} \rbrace. 
\end{align}
All sub-regions $S_{\Omega^{i}}$ corresponding to \eqref{eq-9} together with all switching boundaries $\Sigma_{ij}=\Bar{S}_{\Omega^{i}} \cap \bar{S}_{\Omega^{j}}$ between every pair of successive sub-regions $S_{\Omega^{i}}$ and $S_{\Omega^{j}}$, with $i,j \in \lbrace 1,2, \cdots, 2^{M} \rbrace$, are formally defined in Suppl. sect. \ref{sup}. Note that map \eqref{eq-1-1} is continuous, but has many discontinuities in the Jacobian across the switching boundaries $\Sigma_{ij}$ (for more details please see Suppl. sect. \ref{sup}). It is easy to see that for every pair of matrices $D_{\Omega^k}$ which differ only in one diagonal entry, their corresponding matrices $W_{\Omega^k}$ will differ in only one column.\vspace{3mm}\\ 
\subsection{Transformation from discrete to continuous-time}\label{subsec-2}
As noted in the Introduction, measurements of physical or biological systems are always carried out at discrete times separated by finite time steps $\Delta t$ (often with a constant sampling rate), and efficient algorithms are available for inferring discrete RNN models from such data \cite{dur1,kop219,raz,zha}. However, the state of natural dynamical systems is usually better described in continuous time, and - furthermore - continuous-time formulations enjoy a number of key advantages \cite{che,has}: First, since for ODE systems trajectories are continuous curves rather than collections of single points, some dynamical properties can be determined more easily. For instance, in discrete-time systems it is not possible to distinguish quasiperiodic from periodic orbits with a large period. Likewise, for ODE systems we have a continuous phase portrait (almost) everywhere, and solutions are defined for any arbitrary time point. Finally, some types of analysis are much easier to do in continuous rather than discrete time. For example, performing a change of variables to compute probability distributions within normalizing flows can be more convenient in continuous- rather than discrete-time systems \cite{che}. \footnote{According to \cite{che}, while for discrete systems defined by a bijective map, say $F$, the change in densities due to the mapping by $F$ is given by the determinant of the Jacobian of $F$, for continuous systems the derivative of the logarithm of the probability density with respect to time is given by the \emph{trace} of the Jacobian, which is easier and numerically more robustly to compute.}

In the following we will state a set of theorems which show how to convert a discrete into a continuous PLRNN by transforming discrete-time system \eqref{eq-9} on every sub-region into an equivalent continuous-time system. In this way, a system of piecewise ordinary differential equations will be assigned to the discrete-time system \eqref{eq-9} on $\mathbb{R}^{M}$. Here we will just state our major results, while all details of the proofs will be given in Suppl. sect. \ref{sup}. Note that the following theorems settle the problem for one time step $\Delta t$ taken by the PLRNN, from which, however, results for more time steps immediately follow.
\begin{theorem}\label{thm-1}
Consider discrete-time system \eqref{eq-9} on $S_{\Omega^{k}}$,  $k \in \lbrace 1,2, \cdots, 2^{M} \rbrace$, i.e. the system
\begin{align}\label{equation-1}
& Z_{t+1} \, = \,  F(Z_{t})\, = \, W_{\Omega^{k}} \, Z_{t} + \, h, 
\end{align}
with
\begin{align}\label{}
& W_{\Omega^{k}}:= A +W D_{\Omega^{k}}, \, \, Z_{t} \in S_{\Omega^{k}},  
\end{align}
and time step $\Delta t$. Suppose that $ W_{\Omega^{k}}$ is invertible and has no eigenvalue equal to one, i.e. $P_{W_{\Omega^{k}}}(1) \neq 0$, where $P_{W_{\Omega^{k}}}$ denotes the characteristic polynomials of $W_{\Omega^{k}}$.
\begin{enumerate}
\item[(1)] There exists a continuous-time system 
\begin{align}\label{equation-2}
 \Dot{\zeta} = G(\zeta) = \Tilde{W}_{\Omega^{k}} \, \zeta(t) + \, \Tilde{h}, 
\end{align}
which is equivalent to \eqref{equation-1} on $[t_0, \, \, t_0+\Delta t]$ in the sense that
\begin{align}\label{equation-4-*}
 Z_{t_0}=\zeta(t_0),  \hspace{.3cm} Z_{t_0+\Delta t} = W_{\Omega^{k}} Z_{t_0} + h =\zeta(t_0 +\Delta t).   
\end{align}
Moreover, in this case $\Tilde{W}_{\Omega^{k}}$ is also an invertible matrix and has no eigenvalue equal to one, i.e. $P_{\Tilde{W}_{\Omega^{k}}}(1) \neq 0$. Also,
\begin{align}\label{equation-3}
\begin{cases}
& \Tilde{W}_{\Omega^{k}} \, = \, \frac{1}{\Delta t} \,  log(W_{\Omega^{k}})   
\\[1ex]
&  \Tilde{h} \, = \,  -\frac{1}{\Delta t} \,  log(W_{\Omega^{k}}) \,  \big[I-W_{\Omega^{k}} \big]^{-1} \, \, h
\end{cases}.
\end{align}
Furthermore, if for $ W_{\Omega^{k}}$ each of its Jordan blocks associated with a negative eigenvalue occurs an even number of times, then $\Tilde{W}_{\Omega^{k}}$ will be a real matrix. \\
\item[(2)] If $ W_{\Omega^{k}}$ is both invertible and diagonalizable, 
then $\Tilde{W}_{\Omega^{k}}$ will be invertible and diagonalizable too. 
\end{enumerate}
\end{theorem}
\begin{proof}
See Suppl. sect. \ref{sup} (subsection \ref{p-thm1}). 
\end{proof}
\begin{corollary}\label{cor-1}
The results of theorem \ref{thm-1} are also true if $ W_{\Omega^{k}}$ is a positive-definite matrix with $P_{W_{\Omega^{k}}}(1) \neq 0$. 
\end{corollary}
\begin{proof}
Let $ W_{\Omega^{k}}$ be a positive-definite matrix. Then its determinant is positive, which implies that it is invertible, thus satisfying the conditions of theorem \ref{thm-1}. Note that if $ W_{\Omega^{k}}$ is also Hermitian (or symmetric for real matrices), all eigenvalues of $ W_{\Omega^{k}}$ are real and it is diagonalizable as well \cite{bha}.   
\end{proof}
In theorem \ref{thm-1} it is assumed that $ W_{\Omega^{k}}$ has no eigenvalue equal to one. But there are some PLRNNs with interesting computational properties in the form of system \eqref{equation-1}, for which $ W_{\Omega^{k}}$ has at least one eigenvalue equal to one (see Example \ref{ex-2}). Hence, more generally we are also interested in converting such neural networks from discrete- to continuous-time. The next two theorems are stated to address this problem.  
\begin{theorem}\label{thm-2}
Consider system \eqref{equation-1} and assume that $ W_{\Omega^{k}}$ is invertible, diagonalizable, and has at least one eigenvalue equal to $1$, i.e. $P_{W_{\Omega^{k}}}(1) = det(I -W_{\Omega^{k}}) =0$. Then, there exists an equivalent, in the sense defined in equation \eqref{equation-4-*}, continuous-time system \eqref{equation-2} for \eqref{equation-1} on $[t_0, \, \, t_0+\Delta t]$ such that $\Tilde{W}_{\Omega^{k}}$ is diagonalizable, but not invertible, and
\begin{align}\label{equation-18}
\begin{cases}
& \Tilde{W}_{\Omega^{k}}  \, = \, \frac{1}{\Delta t} \,  log(W_{\Omega^{k}})   
\\[3ex]
&\Tilde{h}  \, = \, -\frac{1}{\Delta t}  \left[ 
\begin{pmatrix}
I_n &   0 \\[1ex]
0 &  \text{O}
\end{pmatrix}
+
log(W_{\Omega^{k}}) \right]
\\[1ex]
& \hspace{1cm}\times
\left[
\begin{pmatrix}
\text{O}_{n \times n} & &  0 \\[1ex]
0 & &  I 
\end{pmatrix}
-
W_{\Omega^{k}}
\right]^{-1} \, h
\end{cases},
\end{align}
where $n$ represents the number of eigenvalues of $W_{\Omega^{k}}$ which are equal to $1$ (or the number of eigenvalues of $ \Tilde{W}_{\Omega^{k}}$ equal to zero).
Also, if each Jordan block of $W_{\Omega^{k}}$ associated with a negative eigenvalue occurs an even number of times, then $\Tilde{W}_{\Omega^{k}}$ will be real. 
\end{theorem}
\begin{proof}
See section \ref{sup} (subsection \ref{p-thm2}).
\end{proof}
\begin{remark}\label{rem-3}
For $n=0$ in \eqref{equation-18}, we have
\begin{align}\label{equation-31}
\begin{pmatrix}
I_n & & 0\\[1ex]
0 & & \text{0}
\end{pmatrix}
= \text{0}, \hspace{1cm}
\begin{pmatrix}
\text{0}_{n \times n} & & 0 \\[1ex]
0 & & I 
\end{pmatrix}
= I.
\end{align}
Thus for $n=0$, i.e. when $W_{\Omega^{k}}$ has no eigenvalue equal to $1$, relations \eqref{equation-3} and \eqref{equation-18} become identical. 
\end{remark}
$W_{\Omega^{k}}$ in theorem \ref{thm-2} must be diagonalizable, but for some computationally interesting PLRNNs $W_{\Omega^{k}}$ is not diagonalizable. The following theorem is stated and proved to address this issue. 
\begin{theorem}\label{thm-3}
Let $ W_{\Omega^{k}}$ in system \eqref{equation-1} be invertible and $P_{W_{\Omega^{k}}}(1) = det(I -W_{\Omega^{k}}) =0$. Then, there exists an equivalent, in the sense of equation \eqref{equation-4-*}, continuous-time system \eqref{equation-2} for \eqref{equation-1} on $[t_0, \, \, t_0+\Delta t]$ such that $\Tilde{W}_{\Omega^{k}}$ is not invertible, and
\begin{align}\label{equation-32}
\begin{cases}
\Tilde{W}_{\Omega^{k}}  \, = \, \frac{1}{\Delta t} \,  log(W_{\Omega^{k}})   
\\[1ex]
  \Tilde{h}  \, = \, \left( W_{\Omega^{k}} \, \, \int_{0}^{\Delta t} e^{-\frac{  \tau}{\Delta t} \,  log(W_{\Omega^{k}})} \,\, \,  d\tau \right)^{-1}\, \, h
\end{cases}.
\end{align}
Further suppose that each Jordan block of $ W_{\Omega^{k}}$ associated with a negative eigenvalue occurs an even number of times, then $\Tilde{W}_{\Omega^{k}}$ is a real matrix.
\end{theorem}
\begin{proof}
See Suppl. section \ref{sup} (subsection \ref{p-thm3}).
\end{proof}
\section{Application examples}\label{sec-4}
In the following we will illustrate how to use our mathematical results for four specific PLRNN systems of relevance in dynamical systems theory and machine learning. These include examples for a nonlinear oscillator (limit cycle), the 'addition problem' introduced by \cite{hoch} to probe 
long short-term-memory capacities of RNN, an example of a chaotic system (Lorenz attractor), and a PLRNN inferred from empirical (human fMRI) time series. Matlab code for all these examples is available at github.com/DurstewitzLab/contPLRNN. 
\begin{example}\label{ex-1}
Consider a discrete-time PLRNN emulation of the nonlinear van-der-Pol oscillator, derived by training a discrete PLRNN with $M=10$ units on time series generated by the van-der-Pol equations (taken from \cite{kop219}, provided online at github.com/DurstewitzLab). The Jacobian matrix of this system is always invertible and has no eigenvalue equal to one in any of the sub-regions $S_{\Omega^{i}}$. Hence we can use theorem \ref{thm-1} to convert this system from discrete- to continuous-time. Fig. \ref{figure1}A illustrates time graphs overlaid for the discrete (blue circles) and continuous (red curves) PLRNN, while Fig. \ref{figure1}B depicts a 2d section of the system's continuous phase space with corresponding flow field. Note that there is a perfect agreement between the discrete and continuous solutions for the set of times at which discrete-PLRNN outputs are defined, while at the same time the continuous PLRNN smoothly interpolates between the discrete-time values. Also note that this agreement continues across different subregions $S_{\Omega^{k}}$ induced by the ReLU function.

\begin{figure*}[hbt!]
\centering
\includegraphics[scale=0.43]{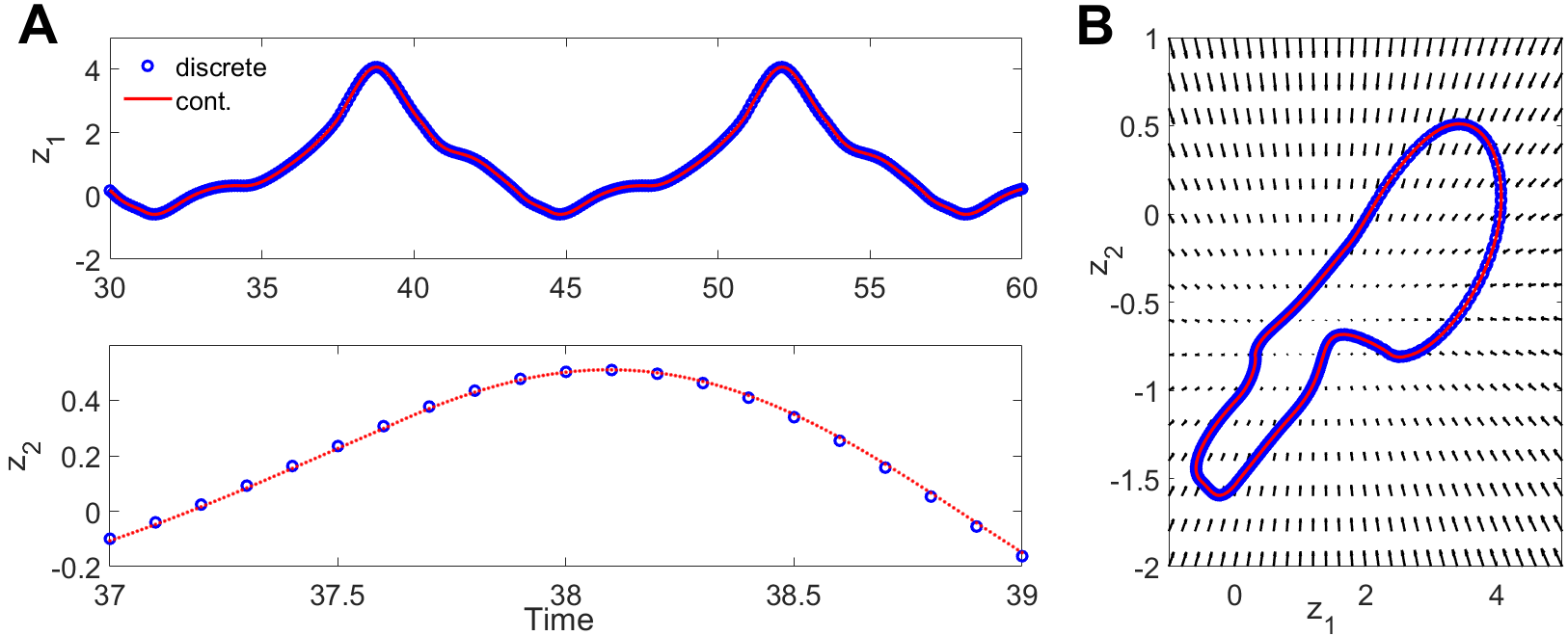}
\caption{Transformation of a discrete PLRNN emulating the nonlinear van-der-Pol oscillator into a continuous-time ODE system. A) Time graphs for two of the system's 10 variables (unit activations). A zoom-in is provided for $z_2$ to better highlight how the continuous solution interpolates between the discrete time points. Blue = discrete PLRNN, red = continuous PLRNN. B) Continuous 2d-subspace of the 10-dimensional state space corresponding to the two variables shown in A, with flow fields (black arrows) and the system's trajectory on the limit cycle (red = continuous, blue circles = discrete); note that since this is only a 2d section of a 10-variable system, convergence to the limit cycle cannot be fully assessed from the $(z_1,z_2)$ vector field.}\label{figure1}
\end{figure*}

As an example of a specific DS analysis that is much easier in the continuous than in the discrete time system we consider a special type of bifurcation (i.e., a point in the system's parameter space where the dynamic abruptly changes), the so-called grazing bifurcation of periodic orbits. It occurs in piecewise smooth continuous-time systems when a periodic orbit tangentially intersects ('grazes') with a switching boundary. A related bifurcation, the so-called border-collision bifurcation, also occurs in discrete-time systems when a $k$-cycle collides with one border. However, in the discrete case, finding the specific bifurcation point can be very challenging, especially in high dimensions, since it amounts to solving highly nested nonlinear equations of the general form $F^k(Z,b)-Z=0$ for large $k$ (where $F^k$ is the $k$-times iterated map and $b$ a bifurcation parameter), 
and determining among the solutions that particular point that agrees with the conditions of the bifurcation.  
In the continuous case, in contrast, one can relatively straightforwardly solve the implied system of equations (see Suppl. section \ref{sup} (subsection \ref{grazing}) for more details), and hence converting the discrete into the continuous PLRNN offers a big advantage. An example for a grazing bifurcation in the continuous-time PLRNN emulation of the van-der-Pol system is shown in Fig. \ref{figure1-2}. Bifurcation phenomena like these are of great practical importance and may also have fundamental implications for training RNN \cite{doy}, since they imply a sudden switch in the temporal structure of the system's behavior as the bifurcation point is crossed.

\begin{figure*}[hbt!]
\centering
\includegraphics[scale=0.505]{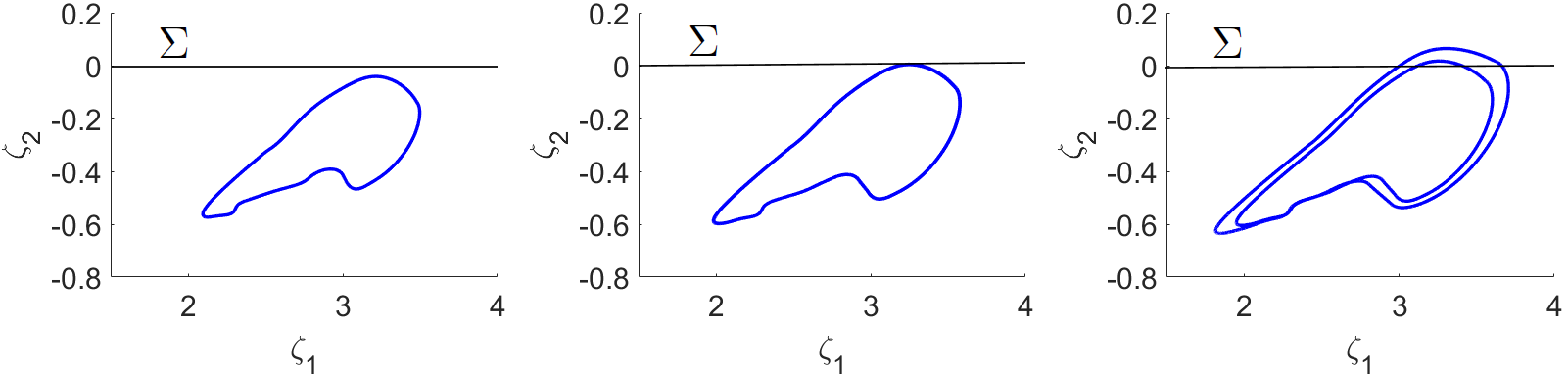}
\caption{Grazing bifurcation in the continuous PLRNN derived from the van-der-Pol oscillator (Example \ref{ex-1}) in the $(\zeta_1, \zeta_2)$ subspace. The system undergoes a bifurcation as weight parameters $\Tilde{w}_{21}^{(1)}$ and $\Tilde{w}_{21}^{(2)}$ 
are decreased from $\Tilde{w}_{21}^{(1)}> \Tilde{w}_{bif}^{(1)}$ and $\Tilde{w}_{21}^{(2)}> \Tilde{w}_{bif}^{(2)}$ to $\Tilde{w}_{21}^{(1)}< \Tilde{w}_{bif}^{(1)}$ and $\Tilde{w}_{21}^{(2)}< \Tilde{w}_{bif}^{(2)}$, where the system's trajectory tangentially touches the border $\Sigma$ (center panel) and the behavior changes from simple-periodic (left panel) to a period-2 limit cycle (right panel).
}\label{figure1-2}
\end{figure*}
\end{example}
\begin{example}\label{ex-2}
 Here we consider a 2-unit RNN solution (adapted from \cite{sch}) to the 'addition problem' introduced in \cite{hoch}. The RNN receives two streams of inputs, one stream of uniform random numbers $s_{1t} \in [0,1]$, and one series of indicator bits $s_{2t} \in \{0,1\}$ which are mostly 0 except for two 6-step time intervals $[t_1,t_1+5]$ and $[t_2,t_2+5]$ where $s_{2,t_1:t_1+5}=s_{2,t_2:t_2+5}=1$. The network's task is to produce as an output the sum of all the inputs in $s_1$ that correspond to the two time intervals $[t_1,t_1+5]$ and $[t_2,t_2+5]$. A simple discrete-time 2-unit PLRNN which (approximately) solves this task is the one with parameters 
 \begin{align}\label{2-unit-1}
A= \begin{pmatrix}
1 &  0\\[1ex]
0 & 0.01
\end{pmatrix}, \hspace{.5cm} W= \begin{pmatrix}
0 &  1\\[1ex]
0 & 0
\end{pmatrix},\hspace{.5cm} h=\begin{pmatrix}
0 \\[1ex] -0.995
\end{pmatrix}.
 \end{align}
 Applying definition \eqref{eq-8}, $W_{\Omega^k}:= A +W D_{\Omega^k}$, we have
 \begin{align}\nonumber
& W_{\Omega^1}= W_{\Omega^2} = \begin{pmatrix}
1 &  1\\[1ex]
0 & 0.01
\end{pmatrix},
\\[1ex]
& W_{\Omega^3}=W_{\Omega^4}= \begin{pmatrix}
1 &  0\\[1ex]
0 & 0.01
\end{pmatrix}.
 \end{align}
Hence, every $W_{\Omega^k}$ is invertible and diagonalizable, but has one eigenvalue equal to one, satisfying the conditions of theorem \ref{thm-2}. Translating this system into continuous time using theorem \ref{thm-2}, Fig. \ref{figure3}A displays the two system variables in both continuous and discrete time. Note that while $z_2$ directly responds to the random inputs, its activity will be integrated (summed) by $z_1$ whenever the second (indicator) inputs are on, i.e. $s_2=1$, which is the case here for the two temporal intervals $[100,105]$ and $[400,405]$, as can be seen by $z_2$ crossing the black dashed line (i.e., when $z_2>0$). Visualizing the continuous system's phase space gives some insight into how the RNN solved the addition problem (Fig. \ref{figure3}B): The $z_1$-line forms a line attractor for $z_2 = -0.995/(1-0.01)$, with the flow converging toward this line from all directions, and zero flow right on this line, thus yielding a dimension in state space along which arbitrary values could be stored \cite{dur2,seu}. A series of sufficiently strong inputs to the RNN (i.e., whenever $s_2=1$) will push the system away from its current to a new position on the line attractor where it will remain until the second supra-threshold series of inputs arrives, in this manner integrating the $s_1$ inputs accompanied by 1's in $s_2$ (see also \cite{sch}). In this specific example, the first series of $s_1$ inputs sums up to $\approx 2.7$ (as marked by the left green circle on the $z_1$-line at $z_2 \approx -1$) and the second to $\approx 3.4$, and the PLRNN correctly reports the total sum of inputs (right green circle) in its final position on the line attractor.
\begin{figure*}[hbt!]
\centering
\includegraphics[scale=0.38]{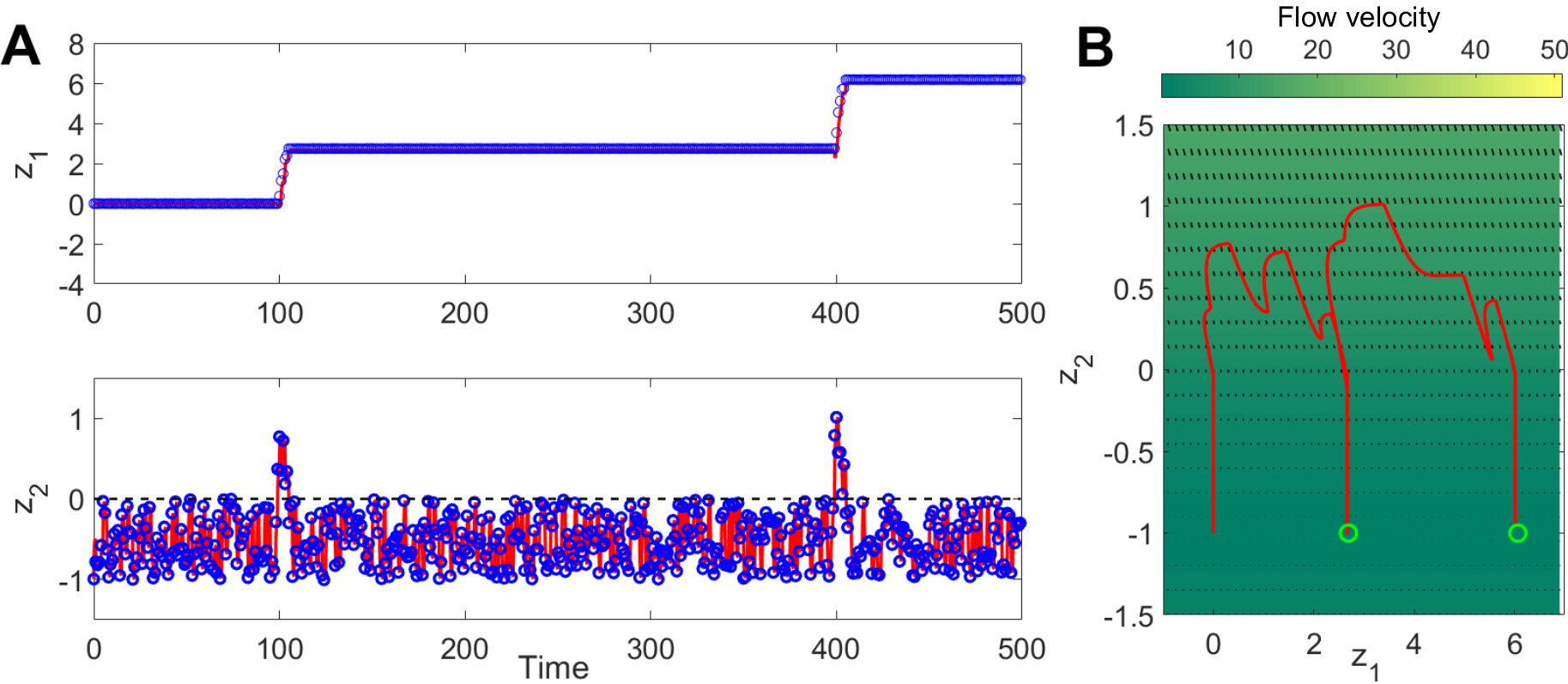}
\caption{Discrete and corresponding continuous time PLRNN solution to the addition problem \cite{hoch}. A) Time graphs for the 2 units of the discrete (blue circles) and continuous (red curves) PLRNN, where $z_1$ sums up the inputs conveyed through $z_2$ whenever its activity crosses 0 (dashed black line). B) State space of continuous-time PLRNN with flow field (black arrows) and trajectory (red) on the addition task. Color coding indicates magnitude of flow (vector length; lighter colors = steeper gradients). Green circles mark the true sums of inputs $s_1$ after the first and second time interval where $s_2=1$. Note that the system's final state on the $z_1$-axis correctly reports the total sum.}\label{figure3}
\end{figure*}
\end{example}
\begin{figure*}[hbt!]
\centering
\includegraphics[scale=0.37]{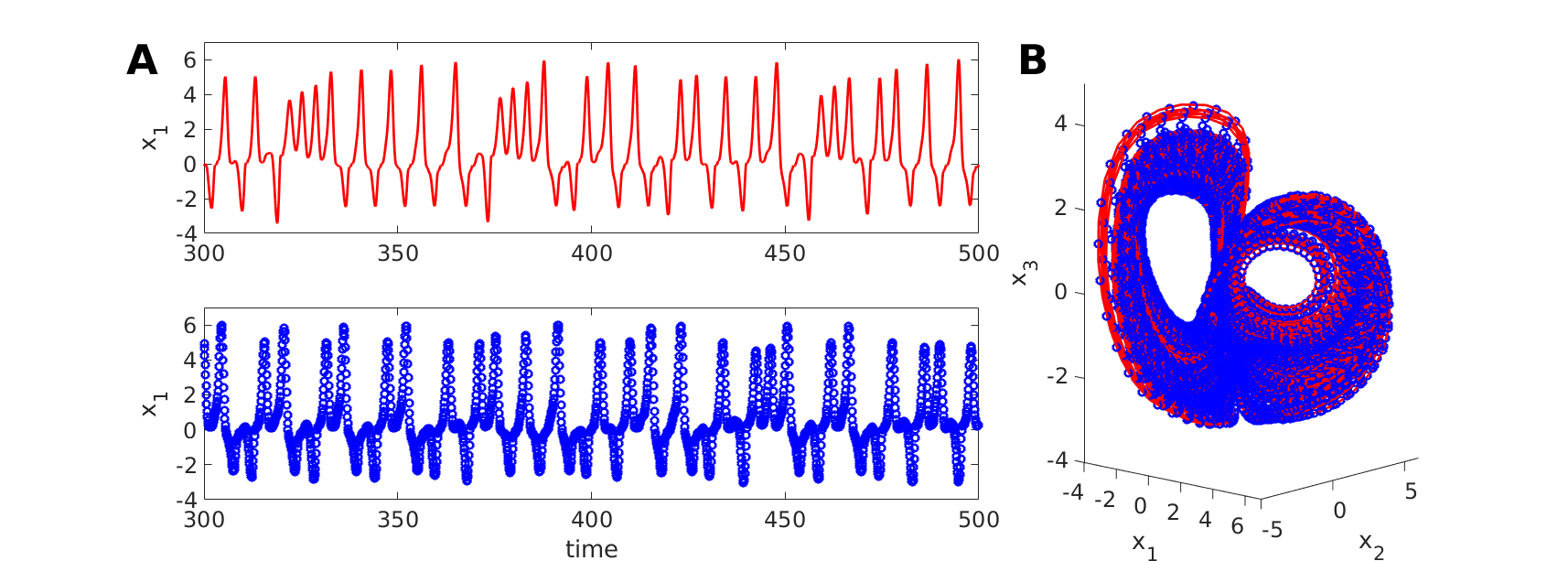}
\caption{Transformation of a discrete PLRNN emulating the Lorenz equations within the chaotic regime into a continuous-time ODE system. A) Time graphs of one of three observation variables produced by projecting the system's 10 latent variables (unit activations) into a 3d observation space. Top: Continuous time solution (red); bottom: discrete time solution (blue). 
Note that since this system is chaotic, time graphs will never precisely overlap since any small numerical difference will lead to exponential divergence of trajectories (here, the recurrence matrices were only close to real for the continuous system and imaginary parts were set to 0). 
B) 3d-projection of the 10-dimensional PLRNN state space corresponding to the three observation variables, exposing the ‘butterfly-wing-type’ structure of the chaotic Lorenz attractor in discrete (blue circles) and continuous (red lines) time.}\label{figure5}
\end{figure*}
\begin{example}\label{ex-3}
As an example for a system with chaotic dynamics we chose a PLRNN emulation ($M=10$) of the 3d Lorenz system, i.e. a PLRNN trained to reproduce the dynamics of the Lorenz equations within the chaotic regime (taken from \cite{kop219}). In this case, we could apply theorem \ref{thm-1} to accomplish the transformation to continuous time, as all matrices $W_{\Omega^k}$ as defined in \eqref{eq-8} were invertible with no eigenvalue equal to 1. Fig. \ref{figure5} confirms that the continuous-time PLRNN agrees with the discrete-time solution also in this case of chaotic behavior.
\end{example}
\begin{figure*}[hbt!]
\centering
\subfigure{\includegraphics[scale=0.472]{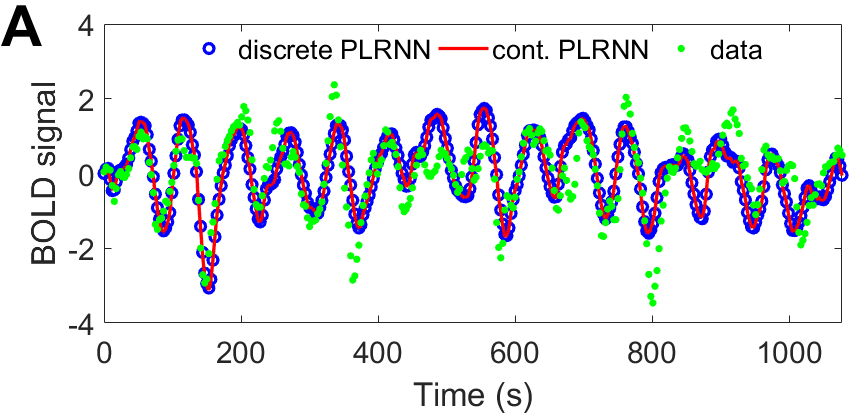}}
\subfigure{\includegraphics[scale=0.472]{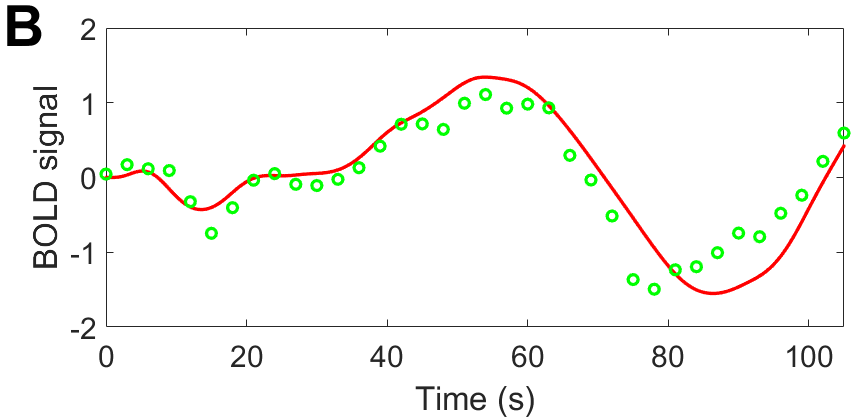}}
\caption{Application of continuous-time transform of PLRNN to human fMRI data. Blue circles in A) show predictions generated from a trained discrete PLRNN (taken from \cite{kop219}) started from an initial condition inferred from the data (green circles). Red curves show smooth inter-/extrapolations produced by the equivalent continuous-time PLRNN. B) provides a zoom into the initial prediction period, where continuous-PLRNN extrapolations were produced at 10 times finer resolution than provided by the experimental data (one scan obtained every 3s).}\label{figure6}
\end{figure*}
\begin{example}\label{ex-4}
As a final example we applied the continuous-time transform developed here to a PLRNN inferred from empirical time series, namely human fMRI data. A discrete PLRNN with $M=10$ latent states was used for this purpose that had been trained on multivariate ($N=20$) time series of Blood Oxygenation Level Dependent (BOLD) signals recorded from human subjects performing a cognitive task while lying in a fMRI scanner. Details on the experimental procedure and task and on PLRNN training can be found in \cite{kop219} (briefly, an Expectation-Maximization algorithm was used for PLRNN training to maximize the evidence-lower-bound [ELBO] of the data log-likelihood).
Theorem \ref{thm-1} could be applied in this case to translate the discrete-time system to continuous time, as all system's matrices met the conditions of this theorem. 
As Fig. \ref{figure6}A demonstrates, the continuous-time PLRNN (red curves) smoothly interpolates between the predictions produced by the discrete PLRNN (blue circles), and thus smoothly extrapolates among the observed data points when started from an initial condition inferred from the experimental data (Fig. \ref{figure6}B, green circles). 
\end{example}
\section{Conclusions}\label{sec-conc}
The aim of the present article was to show how to convert discrete-time into mathematically equivalent continuous-time RNN. As pointed out in the Introduction, sect. \ref{sc-intro}, and sect. \ref{subsec-2}, while RNN are usually trained in discrete time, a continuous-time description enjoys multiple advantages when it comes to analyzing the inferred systems and linking them to scientific theories. Here we examined such a transformation for a particular class of RNN based on rectified-linear unit (ReLU) activation functions. ReLU transfer functions are by now the most common choice in the deep learning community due to their piecewise constant gradients \cite{goo,lin,mon}, easing gradient-descent based algorithms. ReLU-based RNN are computationally and dynamically universal \cite{kim,koi,zhou}, can outperform LSTMs on long short-term memory problems \cite{le,sch}, and can be efficiently inferred from time series data for approximating the underlying (nonlinear) dynamical system \cite{kop219}. Hence our focus on the class of piecewise-linear RNN is not very restrictive, but instead encompasses a whole powerful family of RNN architectures and algorithms. Specifically, we proved that such a conversion from discrete to continuous time is possible under a variety of conditions. These include situations where one or more of the eigenvalues of the system’s Jacobian are equal to 1, as required for long short-term maintenance (Example \ref{ex-2}), or when we have cycles in the discrete case, leading to complex eigenvalues in the continuous case, as we illustrated in our specific applicational examples. Future work may address the problem of discrete-to-continuous-time conversion more generally, for arbitrary nonlinear activation functions, or may attempt to find useful solutions for those cases where a direct translation is not easily possible, e.g. when the recurrence matrix $W_{\Omega^k}$ is not invertible, when the ODE system corresponding to the recursive map would need to have different dimensionality (as may occur, e.g., for low-dimensional chaotic maps like the the logistic map), or when the time step $\Delta t$ is not constant. Another valuable extension, especially in the context of latent variable models, would be to \textit{stochastic} DS.
\section{ Acknowledgements}
This work was funded by the German Science Foundation (DFG) through individual grant Du 354/10-1 to DD, and via the Excellence Cluster 'Structures' at Heidelberg University (EXC-2181 – 390900948). We thank Dr. Georgia Koppe for kindly lending us the fMRI data and corresponding PLRNN parameters used for the empirical example in Fig.\ref{figure6}.

\nocite{langley00}

\bibliography{references}
\bibliographystyle{icml2020}

\newpage
\section*{}
\newpage
\beginsupplement
\section{Supplementary material}\label{sup}
%
\subsection{Sub-regions corresponding to system \eqref{eq-9}}
All sub-regions related to \eqref{eq-9}
can be defined as follows \cite{mon2}:
\begin{align}\label{eq-2}
& S_{\Omega^{1}}= \hat{S}_{0}=\hat{S}_{(\underbrace{0 \, 0 \, 0 \, \cdots\, 0}_{M})^{\ast}_2}= \hat{S}_{\underbrace{0 \, 0\, 0\, \cdots\, 0}_{M}} 
\\[1ex]\nonumber
& \hspace{.5cm}
\ = \ \Big\lbrace Z_ t \in \mathbb{R}^{M}; z_{it}\leq 0, \, i=1,2, \cdots, M \Big\rbrace,
\\[1ex]\label{eq-3}  
& S_{\Omega^{2}}= \hat{S}_1=\hat{S}_{(\underbrace{0\, 0\, \cdots\, 0\, 1}_{M})^{\ast}_2}= \hat{S}_{\underbrace{1\, 0\, 0\, \cdots\, 0}_{M}} 
\\[1ex]\nonumber
& \hspace{.5cm}
\ = \ \Big\lbrace Z_ t \in \mathbb{R}^{M}; z_{1t}>0, z_{it}\leq 0, \, i\neq 1 \Big\rbrace,
\\[1ex]\label{eq-4}  
& S_{\Omega^{3}}=\hat{S}_2=\hat{S}_{(\underbrace{0\, \cdots\, 0\, 1\, 0}_{M})^{\ast}_2}= \hat{S}_{\underbrace{0\, 1\, 0\, \cdots\, 0}_{M}} 
\\[1ex]\nonumber
& \hspace{.5cm}
\ = \ \Big\lbrace Z_ t \in \mathbb{R}^{M}; z_{2t}>0, z_{it}\leq 0, \, i\neq 2 \Big\rbrace,
\\[1ex]\label{eq-5}  
& S_{\Omega^{4}}=\hat{S}_3=\hat{S}_{(\underbrace{0\, \cdots\, 0\, 1\, 1}_{M})^{\ast}_2}= \hat{S}_{\underbrace{1\, 1\, 0\, \cdots\, 0}_{M}}
\\[1ex]\nonumber
& \hspace{.5cm}
\ = \ \Big\lbrace Z_ t \in \mathbb{R}^{M};  z_{1t}, z_{2t}>0, z_{it}\leq 0, \, i\neq 1, 2 \Big\rbrace,
\\[1ex]\label{eq-6}  
& S_{\Omega^5}=\hat{S}_4=\hat{S}_{(\underbrace{0\, \cdots\, 1\, 0\, 0}_{M})^{\ast}_2}= \hat{S}_{\underbrace{0\, 0\, 1\, 0\,  \cdots\, 0}_{M}} 
\\[1ex]\nonumber
& \hspace{.5cm}
\ = \ \Big\lbrace Z_ t \in \mathbb{R}^{M}; z_{3t}>0, z_{it}\leq 0, \, i\neq 3 \Big\rbrace,
\\\nonumber 
& \vdots  \hspace{7.3cm}  \vdots
\\\label{eq-7} 
&S_{\Omega^{2^{M}}}=\hat{S}_{2^{M} \, - \,1}=\hat{S}_{(\underbrace{1\, 1\, 1 \cdots\, 1}_{M})^{\ast}_2}= \hat{S}_{\underbrace{1\, 1\, 1 \cdots \, 1}_{M}}
\\[1ex]\nonumber
& \hspace{.5cm}
\ = \ \Big\lbrace Z_ t \in \mathbb{R}^{M};  z_{it}>0, \, i=1,2, \cdots, M \Big\rbrace.
\end{align}
where each subindex $d$ of $\hat{S}$, $0 \leq d \leq 2^M -1$, is associated with a sequence $d_{M}\, d_{M-1} \, \cdots \, d_{2}\, d_{1}$ of binary digits. The notation $(d_{1}\, d_{2} \, \cdots \, d_{M})_2^{\ast}$ in building each corresponding sequence stands for the mirror image of the binary representation of $d$ with $M$ digits. By mirror image here we mean writing digits $d_{1}\, d_{2} \, \cdots \, d_{M}$ from right to left, i.e. $d_{M}\, d_{M-1} \, \cdots d_{2} \, d_{1}$. For example, for $M=2$ there are $4$ sub-regions $S_{\Omega^{k}}$, $k=1,2,3,4$, associated with $4$ matrices $D_{\Omega^k} := \text{diag} \big(d_{2}, d_{1}\big)$, where
$d_{2}\, d_{1}= (d_{1}\, d_{2})_2^{\ast}$ and $d_i \in \{0,1\}$ (Fig. \ref{figure7}).
\begin{figure*}[hbt!]
\centering
\includegraphics[scale=0.53]{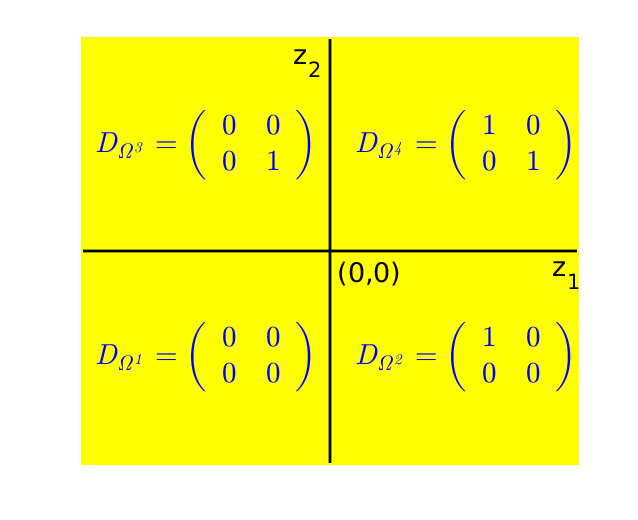}
\caption{Example of subregions $S_{\Omega^{k}}$ and related matrices $D_{\Omega^k}$ for $M=2$.}\label{figure7}
\end{figure*}

Denoting switching boundaries $\Sigma_{ij}=\Bar{S}_{\Omega^{i}} \cap \bar{S}_{\Omega^{j}}$ between every pair of successive sub-regions $S_{\Omega^{i}}$ and $S_{\Omega^{j}}$ with $i,j \in \lbrace 1,2, \cdots, 2^{M} \rbrace$,
we can rewrite map \eqref{eq-1-1} as
\begin{align}\nonumber
 Z_{t+1}& \, = \,  F(Z_{t})
\\[1ex]\label{eq-10}
& =
\begin{cases}
F_1(Z_{t})\, = \, W_{\Omega^1} \, Z_{t} + \, h;  \hspace{1cm} Z_{t} \in \bar{S}_{\Omega^1}
\\[1ex]
F_2(Z_{t})\, = \, W_{\Omega^2} \, Z_{t} + \, h;  \hspace{1cm} Z_{t} \in \bar{S}_{\Omega^2}
\\[1ex]
F_3(Z_{t})\, = \, W_{\Omega^3} \, Z_{t} + \, h;  \hspace{1cm} Z_{t} \in \bar{S}_{\Omega^3}
\\[1ex]
F_4(Z_{t})\, = \, W_{\Omega^4} \, Z_{t} + \, h;  \hspace{1cm} Z_{t} \in \bar{S}_{\Omega^4}
\\ 
 \vdots  \hspace{5cm}  \vdots
\\
F_{2^{M}}(Z_{t})\, = \, W_{\Omega^{2^{M}}} \, Z_{t} + \, h;  \hspace{.5cm} Z_{t} \in \bar{S}_{\Omega^{2^{M}}}
\end{cases}.
\end{align}
\subsection{Discontinuity boundaries}
Consider map \eqref{eq-1-1} and two sub-regions $S_{\Omega^{i}}$ and $S_{\Omega^{j}}$ ($i,j \in \lbrace 1,2, \cdots, 2^{M} \rbrace$) as defined in Section \ref{sec-3} (subsection \ref{subsec-2}). Suppose that subindices $i-1$ and $j-1$ of $\hat{S}_{i-1}$ and $\hat{S}_{j-1}$ are associated with 
$i-1=i_1\, i_2\, \cdots \, i_{M}$ and $j-1=j_1\, j_2\, \cdots \, j_{M}$. Then $S_{\Omega^{i}}$ and $S_{\Omega^{j}}$ are two successive sub-regions with the switching boundary $\Sigma_{ij}=\Bar{S}_{\Omega^{i}} \cap \bar{S}_{\Omega^{j}}$, iff there is exactly one $1 \leq s \leq M$ such that for all $(z_{i_1t}, \cdots, z_{i_{M}t} ) ^T \in \interior{S}_{\Omega^{i}} $ and $(z_{j_1t}, \cdots, z_{j_{M}t} ) ^T \in \interior{S}_{\Omega^{j}} $ 
\begin{align}
\begin{cases}
z_{i_st}\,. \, z_{j_st} <0
\\[1ex]
z_{i_rt}\,. \, z_{j_rt} >0, \, \, 1 \leq \underset{ r \neq s}{r} \leq M
\end{cases}.
\end{align}
Moreover, $\Sigma_{ij}$ is a closed set $(\bar{\Sigma}_{ij} =\Sigma_{ij})$ and $\Sigma_{ij}=\interior{\Sigma}_{ij} \,  \cup \, \partial \Sigma_{ij}$ such that
\begin{align}\nonumber
\interior{\Sigma}_{ij}=\Sigma_{r}^{s} \ = \ \Big\lbrace & Z_ t \in \mathbb{R}^{M}; z_{st}=0, \, \, \text{and} \, \, sgn(z_{rt}) = 
\\[1ex]\label{eq-11}
&sgn(z_{i_rt})= sgn(z_{j_rt}), \, \, 1 \leq \underset{ r \neq s}{r} \leq M \Big\rbrace,  
\end{align}
and $\partial \Sigma_{ij}= \bigcup \limits_{\substack{s_m=1\\s_m \neq s}}^{M}  \Sigma^{s,s_m}_{\nu}$ where
\begin{align}\nonumber
\Sigma^{s,s_m}_{\nu} \ = \ \Big\lbrace & Z_ t \in \mathbb{R}^{M}; z_{s_{m}t}=z_{st}=0, \, \, \text{and} \, \, sgn(z_{\nu t}) 
\\[1ex]
& = sgn(z_{rt}), \, \, 1 \leq \underset{ \nu \neq s, s_m}{\nu} \leq M \Big\rbrace.
\end{align}
Furthermore, it can be proven that
\begin{align}\label{}
 \bigcup\limits_{i,j=1}^{2^M}  \Sigma_{ij} \, = \, \bigcup\limits_{l=1}^{M2^{M-1}}  \Sigma_{l} \subset \bigcup\limits_{k=1}^{2^M}  S_{\Omega^{k}} \, = \,  \mathbb{R}^{M}.
\end{align}
\subsection{Proof of theorem \ref{thm-1}}\label{p-thm1}
$(1)\,$ Without loss of generality let $t_0=0$. Assume that there exists an equivalent continuous-time system for \eqref{equation-1} on $[0, \, \, \Delta t]$, in the form of equation \eqref{equation-2}. By equivalency in the sense of equation \eqref{equation-4-*}, we must have
\begin{align}\label{equation-4}
 Z_0=\zeta(0), \hspace{.5cm} Z_1 = W_{\Omega^{k}} Z_0 + h =\zeta(\Delta t).   
\end{align}
According to theorem \ref{pre-thm-1}, the solution of system \eqref{equation-1} on $[0, \Delta t]$ is 
\begin{align}\label{equation-5}
\zeta(t)=e^{\Tilde{W}_{\Omega^{k}} t} \, \zeta(0) + e^{\Tilde{W}_{\Omega^{k}} t} \int_{0}^{t} e^{-\Tilde{W}_{\Omega^{k}} \tau} \, \, \Tilde{h} \, d\tau,   \hspace{.3cm}  t \in [0, \Delta t]. 
\end{align}
If $\Tilde{W}_{\Omega^{k}}$ is invertible, then
\begin{align}\label{equation-6}
 \int_{0}^{t} e^{-\Tilde{W}_{\Omega^{k}} \tau} \, \, \Tilde{h} \, d\tau = -\Tilde{W}_{\Omega^{k}}^{-1} \big( e^{-\Tilde{W}_{\Omega^{k}} t} -I \big)\, \Tilde{h},
\end{align}
and thus
\begin{align}\nonumber
\zeta(t)=e^{\Tilde{W}_{\Omega^{k}} t} \, \zeta(0) + \Big[& e^{\Tilde{W}_{\Omega^{k}} t} \, \, (-\Tilde{W}_{\Omega^{k}}^{-1}) \, \, e^{-\Tilde{W}_{\Omega^{k}} t} 
\\[1ex]\label{equation-7}
&\, - \, e^{\Tilde{W}_{\Omega^{k}} t} \, \, (- \Tilde{W}_{\Omega^{k}}^{-1}) \Big] \Tilde{h} .
\end{align}
Furthermore, since
\begin{align}\label{equation-8}
(-\Tilde{W}_{\Omega^{k}}^{-1}) \, \, e^{-\Tilde{W}_{\Omega^{k}} t}= e^{-\Tilde{W}_{\Omega^{k}} t} \, \, (-\Tilde{W}_{\Omega^{k}}^{-1}),
\end{align}
we have 
\begin{align}\label{equation-9}
\zeta(t)=e^{\Tilde{W}_{\Omega^{k}} t} \, \zeta(0) + \Big[ I - \,   e^{\Tilde{W}_{\Omega^{k}} t} \Big] (- \Tilde{W}_{\Omega^{k}}^{-1}) \, \, \Tilde{h}  \hspace{.5cm}  t \in [0, \Delta t].
\end{align}
Putting conditions \eqref{equation-4} in 
\begin{align}\label{equation-10}
\zeta(\Delta t)=e^{\Tilde{W}_{\Omega^{k}} \Delta t} \, \zeta(0) + \Big[ I - \,   e^{\Tilde{W}_{\Omega^{k}} \Delta t} \Big] (- \Tilde{W}_{\Omega^{k}}^{-1}) \, \, \Tilde{h},
\end{align}
yields
\begin{align}\label{equation-11}
W_{\Omega^{k}} \, Z_0 + h \, = \, e^{\Tilde{W}_{\Omega^{k}} \Delta t} \, Z_0 + \Big[ I - \,   e^{\Tilde{W}_{\Omega^{k}} \Delta t} \Big] (- \Tilde{W}_{\Omega^{k}}^{-1}) \, \, \Tilde{h}.
\end{align}
Equation \eqref{equation-11} has to hold for all $Z_0$ including $Z_0=0$. Hence, it is deduced that
\begin{align}\label{equation-12}
\begin{cases}
& W_{\Omega^{k}}  \, = \, e^{\Tilde{W}_{\Omega^{k}} \Delta t}
\\[1ex]
& h \, = \, \big[I-e^{\Tilde{W}_{\Omega^{k}} \Delta t} \big]\, (- \Tilde{W}_{\Omega^{k}}^{-1}) \, \, \Tilde{h}
\end{cases}.
\end{align}
According to \eqref{equation-12}, matrix $ W_{\Omega^k}$ should be invertible and cannot have any zero eigenvalue. Also, since $\Tilde{W}_{\Omega^{k}}$ is invertible, it does not have any zero eigenvalue, which implies that $W_{\Omega^{k}}$ has no eigenvalue equal to one. Then, $P_{W_{\Omega^{k}}}(1) \neq 0$, which means $\big[I-W_{\Omega^{k}} \big]$ is invertible and \eqref{equation-12} becomes equivalent to \eqref{equation-3}.  \\

Now, considering $\Tilde{W}_{\Omega^{k}}$ and $\Tilde{h}$ as in \eqref{equation-3}, we can obtain the desired equivalent continuous-time system \eqref{equation-2} for \eqref{equation-1} on $[0, \, \, \Delta t]$.
It is just required to prove that every fixed point $Z^{\ast}$ of map \eqref{equation-1} is also an equilibrium point of system \eqref{equation-2}, and \eqref{equation-3} is  a solution of \eqref{equation-11} for all $Z^{\ast}$. For this purpose, let $Z^{\ast}$ be a fixed point of \eqref{equation-1}, then
\begin{align}\label{equation-13}
F(Z^{\ast})= W_{\Omega^{k}} Z^{\ast} + h \, = \, Z^{\ast}.
\end{align}
$Z^{\ast}$ must be an equilibrium of \eqref{equation-2}, i.e.
\begin{align}\label{equation-14} 
G(Z^{\ast})= \Tilde{W}_{\Omega^{k}} Z^{\ast} + \Tilde{h} \, = \, 0.
\end{align}
From \eqref{equation-13} and \eqref{equation-14} it is concluded that 
\begin{align}\label{equation-15}
h \, = \, \big[I-W_{\Omega^{k}} \big]\, (- \Tilde{W}_{\Omega^{k}}^{-1}) \, \, \Tilde{h}, 
\end{align}
which shows that \eqref{equation-12} or, equivalently, \eqref{equation-3} is a solution of \eqref{equation-11} for all $Z^{\ast}$ satisfying both relations \eqref{equation-13} and \eqref{equation-14}. Finally, let each Jordan block of $ W_{\Omega^{k}}$ associated with a negative eigenvalue occur an even number of times.
Then, by theorem \eqref{pre-thm-4}, the logarithm of real matrix $W_{\Omega^{k}}$, i.e. the matrix $\Tilde{W}_{\Omega^{k}}$ defined in \eqref{equation-3}, will be real.
\\ \\
$(2)\, $ Let $ W_{\Omega^{k}}$ be diagonalizable, then
\begin{align}\label{equation-16}
W_{\Omega^{k}} \, = \, V \, E_{k} \, V^{-1},
\end{align}
where $E_{k}=diag(\lambda_1^{k}, \lambda_2^{k}, \cdots, \lambda_M^{k})$ and V is the matrix of eigenvectors of $W_{\Omega^k}$. Since $ W_{\Omega^{k}}$ is also invertible, by \eqref{equation-3} 
\begin{align}\nonumber
\Tilde{W}_{\Omega^{k}} &\, = \, \frac{1}{\Delta t} \,  log(W_{\Omega^{k}}) \, = \, \frac{1}{\Delta t} \,  log(V \, E_{k} \, V^{-1}) 
\\[1ex]\label{equation-17}
&\, = \, V \frac{1}{\Delta t} \,  log(E_{k}) \, V^{-1},
\end{align}
such that $$log(E_{k}) = diag \big(log(\lambda_1^{k}), log(\lambda_2^{k}), \cdots, log(\lambda_M^{k}) \big),$$ which completes the proof.\\

\textbf{Remark.}
Due to \eqref{equation-12} and \eqref{equation-3}, one can see that  
system \eqref{equation-2} is homogeneous ($\Tilde{h}=0$) if and only if system \eqref{equation-1} is homogeneous ($h=0$). 
\subsection{Proof of theorem \ref{thm-2}}\label{p-thm2}
Again we prove the theorem for $t_0= 0$ without loss of generality. Suppose that there is the equivalent continuous-time system \eqref{equation-2} for \eqref{equation-1} with non-invertible and diagonalizable matrix $\Tilde{W}_{\Omega^{k}}$. Similar to the proof of the previous theorem, relations \eqref{equation-4} and \eqref{equation-5} must hold for \eqref{equation-1} and \eqref{equation-2}. On the other hand, non-invertibility and diagonalizability of $ \Tilde{W}_{\Omega^{k}}$ demand that it has at least one eigenvalue equal to zero and  
\begin{align}\label{equation-19}
\Tilde{W}_{\Omega^{k}} \, = \,  V\, 
\begin{pmatrix}
\text{O}_{n \times n} & & & & 0 \\[1ex]
0 & & & & C
\end{pmatrix}
\, V^{-1},  
\end{align}
where $\text{O}_{n \times n}$ is a zero matrix corresponding to zero eigenvalues ($n$ denotes the number of zero eigenvalues) and $C$ is an invertible matrix corresponding to nonzero eigenvalues of $\Tilde{W}_{\Omega^{k}}$. Therefore, for relation \eqref{equation-6} we obtain
\begin{align}\nonumber
& \int_{0}^{t} e^{-\Tilde{W}_{\Omega^{k}} \tau} \, \, \Tilde{h} \, d\tau = 
\\[1ex]\label{equation-20}
 & V\, 
\begin{pmatrix}
\begin{pmatrix}
t & \cdots & 0 \\
 \vdots &\ddots & \vdots \\
0 &  \cdots & t
\end{pmatrix}_{n \times n}  & &  0 \\[1ex]
0 & &  -C^{-1} \big(e^{-Ct} \, - I \big)
\end{pmatrix}
\, V^{-1} \, \Tilde{h}.
\end{align}
In this case, relation \eqref{equation-10} becomes 
\begin{align}\nonumber
& \zeta (\Delta t)=e^{\Tilde{W}_{\Omega^{k}} \Delta t} \, \zeta(0) +   e^{\Tilde{W}_{\Omega^{k}} \Delta t}\,  V \times
\\[1ex]\label{equation-21}
 & 
\begin{pmatrix}
\begin{pmatrix}
\Delta t & \cdots & 0\\
\vdots &\ddots & \vdots              \\
0 &  \cdots & \Delta t
\end{pmatrix}_{n \times n}  & &  0 \\[1ex]
0 & &  -C^{-1} \big(e^{-C \Delta t} - I \big)
\end{pmatrix}
V^{-1} \Tilde{h}. 
\end{align}
Inserting conditions \eqref{equation-4} into \eqref{equation-21} gives
\begin{align}\nonumber
& W_{\Omega^{k}} \, Z_0 + h \, = \, e^{\Tilde{W}_{\Omega^{k}} \Delta t} \, Z_0 + e^{\Tilde{W}_{\Omega^{k}} \Delta t}\,  V\, \times
\\[1ex]\label{equation-22}
 &
\begin{pmatrix}
\begin{pmatrix}
\Delta t & \cdots & 0\\
 \vdots &\ddots & \vdots              \\
0 &  \cdots & \Delta t
\end{pmatrix}_{n \times n}  & &  0 \\[1ex]
0 & & -C^{-1} \big(e^{-C \Delta t}- I \big)
\end{pmatrix}
V^{-1} \Tilde{h}.
\end{align}
Denoting
 \begin{align}\label{H}
 H  \, = \,  \begin{pmatrix}
\Delta t & \cdots & 0\\
 \vdots &\ddots & \vdots              \\
0 &  \cdots & \Delta t
\end{pmatrix}_{n \times n}, 
\end{align}
and considering equality \eqref{equation-22} for all $Z_0$, particularly for $Z_0=0$, yields
\begin{align}\label{equation-23}
\begin{cases}
& W_{\Omega^{k}} \, = \, e^{\Tilde{W}_{\Omega^{k}} \Delta t} 
\\[1ex]
& h \, = \,e^{\Tilde{W}_{\Omega^{k}} \Delta t}\, \,  V 
\begin{pmatrix}
H &   0 \\[1ex]
0 &   -C^{-1} \big(e^{-C \Delta t} - I \big)
\end{pmatrix}
V^{-1} \Tilde{h} 
\end{cases}.
\end{align}
Since 
\begin{align}\label{equation-24}
e^{\Tilde{W}_{\Omega^{k}} \Delta t} \, = \,  V\, 
\begin{pmatrix}
I & & & & 0 \\[1ex]
0 & & & & e^{C \Delta t} 
\end{pmatrix}
\, V^{-1},  
\end{align}
we can simplify $h$ in \eqref{equation-23} and rewrite it as
\begin{align}\label{equation-25}
\begin{cases}
& W_{\Omega^{k}} \, = \, e^{\Tilde{W}_{\Omega^{k}} \Delta t}
\\[1ex]
& h \, =\,  V  
\begin{pmatrix}
H &  0 \\[1ex]
0 &  -C^{-1} \big( I - e^{C \Delta t} \big)
\end{pmatrix}
V^{-1} \Tilde{h}, 
\end{cases}
\end{align}
or equivalently
\begin{align}\label{equation-26}
\begin{cases}
& \Tilde{W}_{\Omega^{k}}  \, = \, \frac{1}{\Delta t} \,  log(W_{\Omega^{k}})   
\\[1ex]
&\Tilde{h}  \, = \,  V  
\begin{pmatrix}
H & &  0 \\[1ex]
0 & &  -C^{-1} \big(I - e^{C \Delta t} \big)
\end{pmatrix}^{-1}
V^{-1} h.
\end{cases}
\end{align}
In addition, we can write 
%
\begin{align}\nonumber
\Tilde{h} &  =  V 
\begin{pmatrix}
H & &  0 \\[1ex]
0 & & -C^{-1} \big( I -  e^{C \Delta t} \big)
\end{pmatrix}^{-1}
V^{-1}  h 
\\[1ex]\nonumber
& \, = \,  V \, 
\begin{pmatrix}
\begin{pmatrix}
\frac{1}{\Delta t} & \cdots & 0 \\
 \vdots &\ddots & \vdots \\
0 &  \cdots & \frac{1}{\Delta t}
\end{pmatrix}_{n \times n}  &   0 \\[1ex]\nonumber
0 &   C  \big(e^{C \Delta t} - I \big)^{-1}
\end{pmatrix}
\, V^{-1} \, h
\\[1ex]\nonumber
& \, = \,  V  \left[ 
\frac{1}{\Delta t} 
\begin{pmatrix}
I_{n} & &   0\\[1ex]
0 & & \text{O}
\end{pmatrix}
+
\begin{pmatrix}
\text{O}_{n \times n}  & & 0 \\[1ex]
0 & &  C
\end{pmatrix}
\right]
\\[1ex]\nonumber
& \hspace{1cm}\times  \begin{pmatrix}
I_{n} & &  0 \\[1ex]
0 & & \big( e^{C \Delta t} - I \big)^{-1}
\end{pmatrix}
 V^{-1} h
\\[1ex]\nonumber
& \, = \,  V \left[ 
\frac{1}{\Delta t} 
\begin{pmatrix}
I_n & &   0\\[1ex]
0 & & \text{O}
\end{pmatrix}
+
V^{-1} \, \Tilde{W}_{\Omega^{k}} \,V \right] 
\\[1ex]\nonumber
& \hspace{1cm} \times
\left[
\begin{pmatrix}
I_n & &  0 \\[1ex]
0 & &  e^{C \Delta t} 
\end{pmatrix}
-
\begin{pmatrix}
\text{O}_{n \times n}   & 0 \\[1ex]
0 &  I 
\end{pmatrix}
\right]^{-1} V^{-1} h
\\[1ex]\nonumber
& \, = \, \left[ 
\frac{1}{\Delta t} 
\begin{pmatrix}
I_n &  0 \\[1ex]
0 & \text{O}
\end{pmatrix}
+
\Tilde{W}_{\Omega^{k}} \right]
\\[1ex]\label{equation-27}
& \hspace{1cm} \times
\left[
e^{\Tilde{W}_{\Omega^{k}} \Delta t}
-
\begin{pmatrix}
\text{O}_{n \times n} &  0 \\[1ex]
0 &  I 
\end{pmatrix}
\right]^{-1} h.
\end{align}
Therefore
\begin{align}\nonumber
%
%
&\Tilde{W}_{\Omega^{k}}   =  \frac{1}{\Delta t} \,  log(W_{\Omega^{k}}),   
\\[2ex]\nonumber
&\Tilde{h} =  \left[ 
-\frac{1}{\Delta t} 
\begin{pmatrix}
I_{n} &   0 \\[1ex]
0 &  \text{O}
\end{pmatrix}
-
\Tilde{W}_{\Omega^{k}} \right]
\\[1ex]\label{equation-28}
& \hspace{1cm}\times
\left[
\begin{pmatrix}
\text{O}_{n \times n} &  0 \\[1ex]
0 &  I 
\end{pmatrix}
-
e^{\Tilde{W}_{\Omega^{k}} \Delta t}
\right]^{-1}  h
%
\end{align}
which is equivalent to \eqref{equation-18}.\\

Finally, from $ W_{\Omega^{k}} \, = \, e^{\Tilde{W}_{\Omega^{k}} \Delta t}$ it is deduced that $W_{\Omega^{k}}$ is invertible. It is only necessary to prove that for every point $Z^{\ast}$ satisfying both equations \eqref{equation-13} and \eqref{equation-14}, i.e. equations
\begin{equation}\label{equation-29}
\begin{cases}
\big( W_{\Omega^{k}} - I \big) Z^{\ast} \, = \, -h
\\[1ex]
\Tilde{W}_{\Omega^{k}} Z^{\ast}  \, = \, - \Tilde{h}
\end{cases},
\end{equation}
relation \eqref{equation-18} is a solution of \eqref{equation-22}. Note that here we cannot simplify \eqref{equation-29} to find some equation similar to \eqref{equation-15}, as neither $\big( W_{\Omega^{k}} - I \big)$ or $\Tilde{W}_{\Omega^{k}}$ is invertible. Hence, we show that \eqref{equation-29} fulfills solution \eqref{equation-18} or, identically, solution \eqref{equation-25}. Thus, inserting $\Tilde{h} \, = \, - \Tilde{W}_{\Omega^{k}} Z^{\ast}$ in \eqref{equation-25}, we have
\begin{align}\nonumber
h &  = e^{\Tilde{W}_{\Omega^{k}} \Delta t}\, \,  V\, 
\begin{pmatrix}
H &  0 \\[1ex]
0 &  -C^{-1} \big(e^{-C \Delta t}  - I \big)
\end{pmatrix}
 V^{-1}  \Tilde{h} 
\\[1ex]\nonumber
&  = e^{\Tilde{W}_{\Omega^{k}} \Delta t} \, V 
\begin{pmatrix}
H &   0 \\[1ex]
0 &  -C^{-1} \big(e^{-C \Delta t} - I \big)
\end{pmatrix}
 V^{-1} (- \Tilde{W}_{\Omega^{k}} Z^{\ast}) 
\\[1ex]\nonumber
&  = e^{\Tilde{W}_{\Omega^{k}} \Delta t}\,  V
\begin{pmatrix}
H &  0 \\[1ex]
0 &  -\big(e^{-C \Delta t}  - I \big) C^{-1} 
\end{pmatrix}
 V^{-1} \, V
 \\[1ex]\nonumber
& \hspace{1cm} \times
\begin{pmatrix}
\text{O}_{n \times n} &  0 \\[1ex]
0 &  -C
\end{pmatrix}
 V^{-1} Z^{\ast}
\\[1ex]\nonumber
&  = e^{\Tilde{W}_{\Omega^{k}} \Delta t} \,  V\, 
\begin{pmatrix}
\text{O}_{n \times n}  & & 0 \\[1ex]
0 & &  \big(e^{-C \Delta t} \, - I \big)
\end{pmatrix}
\, V^{-1} Z^{\ast}
\\[1ex]\nonumber
& \, = \,e^{\Tilde{W}_{\Omega^{k}} \Delta t}\,  V\, 
\left[
\begin{pmatrix}
I_{n} &  0 \\[1ex]
0 &  e^{-C \Delta t} 
\end{pmatrix}
-
\begin{pmatrix}
I_{n} & & 0 \\[1ex]
0 & &  I
\end{pmatrix}
\right] V^{-1} Z^{\ast}
\\[1ex]\label{equation-30}
& 
 = 
\left(I - e^{\Tilde{W}_{\Omega^{k}} \Delta t} \right) Z^{\ast}
 = 
\left(I - W_{\Omega^{k}} \right) Z^{\ast},
\end{align}
which demonstrates that \eqref{equation-29} meets solution \eqref{equation-25}.\\

If every Jordan block of $W_{\Omega^{k}}$ associated with a negative eigenvalue occurs an even number of times, then theorem \eqref{pre-thm-4} guarantees that $\Tilde{W}_{\Omega^{k}}$ will be real. Also, similar to the proof of theorem \ref{thm-1}, it is easy to see that $\Tilde{W}_{\Omega^{k}}$ will be diagonalizable when $ W_{\Omega^{k}}$ has no negative real eigenvalues.
\subsection{Proof of theorem \ref{thm-3}}\label{p-thm3}
Let $t_0= 0$ without loss of generality and assume there exists the equivalent continuous-time system \eqref{equation-2} for \eqref{equation-1}, for which matrix $ W_{\Omega^{k}}$ is non-invertible. Then, relations \eqref{equation-4} and \eqref{equation-5} must hold for \eqref{equation-1} and \eqref{equation-2}, analogously to the proofs of the previous theorems. Also, by similar reasoning we have 
\begin{align}\label{equation-33}
\zeta(\Delta t)=e^{\Tilde{W}_{\Omega^{k}} \Delta t} \, \zeta(0) + \left[ e^{\Tilde{W}_{\Omega^{k}} \Delta t}\,  \int_{0}^{\Delta t} e^{-\Tilde{W}_{\Omega^{k}} \tau} \, d\tau  \right] \, \Tilde{h} .
\end{align}
Inserting conditions \eqref{equation-4} in equation \eqref{equation-33} and solving the resulting equation for all $Z_0$, including $Z_0=0$, yields
\begin{align}\label{equation-34}
\begin{cases}
& W_{\Omega^{k}} \, = \, e^{\Tilde{W}_{\Omega^{k}} \Delta t} 
\\[1ex]
& h \, = \,e^{\Tilde{W}_{\Omega^{k}} \Delta t}\, \,  \left( \int_{0}^{\Delta t} e^{-\Tilde{W}_{\Omega^{k}} \tau} \, d\tau   \right) \, \Tilde{h}. 
\end{cases}
\end{align}
Now let 
\begin{align}\label{equation-35}
\lambda \in \text{Spectrum}(\Tilde{W}_{\Omega^{k}}) \, \Rightarrow \,  \lambda \Delta t \notin 2i \pi \mathbb{Z}^*.
\end{align}
Then, by proposition \ref{pro-1}, $ \int_{0}^{\Delta t} e^{-\Tilde{W}_{\Omega^{k}} \tau} \, d\tau $ is invertible and so
\begin{align}\label{equation-32-*}
\begin{cases}
\Tilde{W}_{\Omega^{k}}  \, = \, \frac{1}{\Delta t} \,  log(W_{\Omega^{k}})   
\\[1ex]
  \Tilde{h}  \, = \, \left( \int_{0}^{\Delta t} e^{-\Tilde{W}_{\Omega^{k}} \tau} \, d\tau \right)^{-1} \,\,e^{-\Tilde{W}_{\Omega^{k}} \Delta t} \, \, h
\end{cases},
\end{align}
which is equal to equation \eqref{equation-32}. The last point which still has to be proven is that equation \eqref{equation-29} meets solution \eqref{equation-32} or, identically, \eqref{equation-34}, for every $Z^{\ast}$. Since $\Tilde{W}_{\Omega^{k}}$ is non-invertible, it can be written in the following Jordan form:  
\begin{align}\label{equation-36}
\Tilde{W}_{\Omega^{k}} \, = \,  U\,
\begin{pmatrix}
B & & & & 0 \\[1ex]
0 & & & & C
\end{pmatrix}
\, U^{-1},  
\end{align}
where $B$ is a strictly upper triangular matrix and $C$ is an invertible matrix. Then
\begin{align}\label{equation-37}
& e^{-\Tilde{W}_{\Omega^{k}} \Delta t} \, = \,  U\, 
\begin{pmatrix}
e^{-B \Delta t}  & & 0 \\[1ex]
0 & &  e^{-C \Delta t} 
\end{pmatrix}
\, U^{-1},
\\[1ex]\nonumber
& \int_{0}^{\Delta t} e^{-\Tilde{W}_{\Omega^{k}} \tau} \, \Tilde{h} \, d\tau =
\\[1ex]\label{equation-37-2}
& \hspace{1cm} U
\begin{pmatrix}
 \int_{0}^{\Delta t} e^{-B \tau}  d\tau  &  0 \\[1ex]
0 &  -C^{-1} \big(e^{-C\Delta t}  - I \big)
\end{pmatrix}
\, U^{-1}.
\end{align}
Now, substituting $\Tilde{h} \, = \, - \Tilde{W}_{\Omega^{k}} Z^{\ast}$ in \eqref{equation-34} we have
\begin{align}\nonumber
h & \, = \,e^{\Tilde{W}_{\Omega^{k}} \Delta t}\, \,  \left( \int_{0}^{\Delta t} e^{-\Tilde{W}_{\Omega^{k}} \tau} \, d\tau   \right) \, \Tilde{h}
\\[1ex]\nonumber
&  =  e^{\Tilde{W}_{\Omega^{k}} \Delta t}\, U 
\begin{pmatrix}
 \int_{0}^{\Delta t} e^{-B \tau} \, d\tau  &  0 \\[1ex]
0 &  -C^{-1} \big(e^{-C\Delta t} - I \big)
\end{pmatrix}
\\[1ex]\nonumber
& \hspace{1cm} \times
 U^{-1}  (- \Tilde{W}_{\Omega^{k}} Z^{\ast})
\\[1ex]\nonumber
&  = e^{\Tilde{W}_{\Omega^{k}} \Delta t}\,   U 
\begin{pmatrix}
 \int_{0}^{\Delta t} e^{-B \tau} \, d\tau  &  0 \\[1ex]
0 & - \big(e^{-C\Delta t} \, - I \big) \, C^{-1}
\end{pmatrix}
\\[1ex]\nonumber
& \hspace{1cm} \times
U^{-1}\, U
\begin{pmatrix}
-B &  0 \\[1ex]
0 &  -C
\end{pmatrix}
\, U^{-1} Z^{\ast}
\\[1ex]\nonumber
& =e^{\Tilde{W}_{\Omega^{k}} \Delta t}\,   U\, 
\begin{pmatrix}
 \int_{0}^{\Delta t} -B \, e^{-B \tau} \, d\tau  &  0 \\[1ex]
0 &   \big(e^{-C\Delta t} \, - I \big)
\end{pmatrix}
\\[1ex]\nonumber
& \hspace{1cm} \times
 U^{-1} Z^{\ast}
\\[1ex]\nonumber
&  = e^{\Tilde{W}_{\Omega^{k}} \Delta t}\,  U 
\begin{pmatrix}
 \big(e^{-B \Delta t}- I \big) &  0 \\[1ex]
0 &   \big(e^{-C\Delta t} - I \big)
\end{pmatrix}
\, U^{-1} Z^{\ast}
\\[1ex]\label{equation-38}
& = 
\left(I - e^{\Tilde{W}_{\Omega^{k}} \Delta t} \right) Z^{\ast}
 = 
\left(I - W_{\Omega^{k}} \right) Z^{\ast},
\end{align}
which completes the proof.\\

Finally, due to theorem \eqref{pre-thm-4}, $\Tilde{W}_{\Omega^{k}}$ will be real, provided that each Jordan block of $ W_{\Omega^{k}}$ related to a negative eigenvalue occurs an even number of times.\\

\textbf{Remark.}
In theorem \ref{thm-3}, by \eqref{equation-37-2} we have
\begin{align}
& \left( \int_{0}^{\Delta t} e^{-\Tilde{W}_{\Omega^{k}} \tau} \, d\tau \right)^{-1} \, = \,  
 \\[1ex]\label{equation-39}
 &
 U \, 
\begin{pmatrix}
\left(\int_{0}^{\Delta t} e^{-B \tau}  \, d\tau \right)^{-1} \, & &  0 \\[1ex]
0 & & \big(I - e^{-C \Delta t} \big)^{-1} \,  C 
\end{pmatrix}
\, U^{-1}  h. 
\end{align}
On the other hand, since $C$ is invertible, $ \det \big(I - e^{-C \Delta t} \big) \neq 0$. Therefore, $\int_{0}^{\Delta t} e^{-\Tilde{W}_{\Omega^{k}} \tau} \, d\tau $ is invertible if and only if $ \int_{0}^{\Delta t} e^{-B \tau} \, d\tau $ is invertible. Thus, for invertibility of $\int_{0}^{\Delta t} e^{-\Tilde{W}_{\Omega^{k}} \tau} \, d\tau $, it is required that relation \eqref{equation-35} holds only for any pair of eigenvalues of $B$. 
\subsection{Grazing bifurcation}\label{grazing}
Here we investigate a grazing bifurcation of periodic orbits for the continuous PLRNN derived from the van-der-Pol oscillator (Example \ref{ex-1}). For this purpose, we consider the converted continuous-time system locally in the neighborhood of only one border 
\begin{align*}
\Sigma = \Big\{ \zeta=(\zeta_1, \zeta_2, \cdots, \zeta_{10})^T \in \mathbb{R}^{10} \, | \, H(\zeta)=\zeta_2=0  \Big\},
\end{align*}
where the scalar function $H:\mathbb{R}^{10} \rightarrow \mathbb{R}$ defines the border and has non-vanishing gradient. According to \cite{ber,monf}, a periodic orbit $\hat{\zeta}(t)$ undergoes a grazing bifurcation for some critical value of a bifurcation parameter, if it is a grazing orbit for some $t=t^{*}$. This means $\hat{\zeta}(t)$ hits $\Sigma$ tangentially at the grazing point $\hat{\zeta}^{*} = \hat{\zeta}(t^{*})$ and satisfies the following conditions:
\begin{align*}
& H(\hat{\zeta}^{*}) \, = \, \hat{\zeta}^{*}_2=0,
\\[1ex]
& \nabla H(\hat{\zeta}^{*}) \, = \, (0, 1, 0, \cdots, 0 )^T \neq 0,
\\[1ex]
& \big< \nabla H(\hat{\zeta}^{*}), \, \Tilde{W}_{\Omega^{1}} \, \hat{\zeta}^{*} + \, \Tilde{h}_1 \big> \, = \, \sum_{\substack{j=1\\ j \neq 2}}^{10} \Tilde{w}_{2j}^{(1)} \, \hat{\zeta}^{*}_{j} \, + \,  \Tilde{h}_{12} = 0, 
\\[1ex]
& \big< \nabla H(\hat{\zeta}^{*}), \, \Tilde{W}_{\Omega^{2}} \, \hat{\zeta}^{*} + \, \Tilde{h}_2 \big> \, = \, 
\sum_{\substack{j=1\\ j \neq 2}}^{10} \Tilde{w}_{2j}^{(2)} \, \hat{\zeta}^{*}_{j} \, + \,  \Tilde{h}_{22} = 0,
\\[1ex]
%
%
& \big< \nabla H(\hat{\zeta}^{*}), \, \Tilde{W}^2_{\Omega^{1}} \hat{\zeta}^{*} + \Tilde{W}_{\Omega^{1}} \Tilde{h}_1 \big> \, + \, \big< \nabla^2 H(\hat{\zeta}^{*}) (\Tilde{W}_{\Omega^{1}} \, \hat{\zeta}^{*}  
\\
& + \,\Tilde{h}_1), \, \Tilde{W}_{\Omega^{1}} \, \hat{\zeta}^{*} + \,\Tilde{h}_1 \big> = \, \sum_{\substack{j=1\\ j \neq 2}}^{10} \Tilde{v}_{2j}^{(1)} \hat{\zeta}^{*}_{j} +  \sum_{\substack{j=1\\ j \neq 2}}^{10} \Tilde{w}_{2j}^{(1)} \, \Tilde{h}_{1j} = 0,
\\[1ex]
& \big< \nabla H(\hat{\zeta}^{*}), \, \Tilde{W}^2_{\Omega^{2}} \hat{\zeta}^{*} + \Tilde{W}_{\Omega^{2}} \Tilde{h}_2 \big> \, + \, \big< \nabla^2 H(\hat{\zeta}^{*}) (\Tilde{W}_{\Omega^{2}} \, \hat{\zeta}^{*}  
\\
&+ \,\Tilde{h}_2), \, \Tilde{W}_{\Omega^{2}} \, \hat{\zeta}^{*} + \,\Tilde{h}_2 \big> = \, \sum_{\substack{j=1\\ j \neq 2}}^{10} \Tilde{v}_{2j}^{(2)} \hat{\zeta}^{*}_{j} + \sum_{\substack{j=1\\ j \neq 2}}^{10} \Tilde{w}_{2j}^{(2)} \, \Tilde{h}_{2j} = 0,
\end{align*}
where $\Tilde{W}_{\Omega^{1}}= [\Tilde{w}_{ij}^{(1)}]$, $\Tilde{W}_{\Omega^{2}}= [\Tilde{w}_{ij}^{(2)}]$, $\Tilde{W}^2_{\Omega^{1}}= [\Tilde{v}_{ij}^{(1)}]$ and $\Tilde{W}^2_{\Omega^{2}}= [\Tilde{v}_{ij}^{(2)}]$.  \\
In this case the periodic orbit $\hat{\zeta}(t)$ crosses $\Sigma$ transversally as the bifurcation parameter passes through the bifurcation value. The grazing bifurcation leads to a transition or a sudden jump in the system’s response by the dis-/appearance of a tangential intersection between the trajectory and the switching boundary. The occurrence of a grazing bifurcation in the continuous PLRNN is illustrated in Fig. \ref{figure1-2}.

\end{document}